\documentclass[11pt,a4paper]{amsart}
\usepackage{mathrsfs}
\usepackage{syntonly}
\usepackage{amsmath}
\usepackage{amsthm}
\usepackage{amsfonts}
\usepackage{amssymb}
\usepackage{latexsym}
\usepackage{amscd,amssymb,amsopn,amsmath,amsthm,graphics,amsfonts,mathrsfs,accents,enumerate,verbatim,calc}
\usepackage[dvips]{graphicx}
\usepackage[colorlinks=true,linkcolor=red,citecolor=blue]{hyperref}
\usepackage[all]{xy}
\usepackage{enumitem}

\date{}
\allowdisplaybreaks[4] \footskip=15pt
\renewcommand{\uppercasenonmath}[1]{}

\numberwithin{equation}{section} \theoremstyle{plain}
\newtheorem{lem}{Lemma}[section]
\newtheorem{cor}[lem]{Corollary}
\newtheorem{prop}[lem]{Proposition}
\newtheorem{thm}[lem]{Theorem}

\newtheorem{cond}[lem]{Condition}
\newtheorem{definition}[lem]{Definition}
\newtheorem{Ex}[lem]{Example}
\newtheorem{Quest}[lem]{Question}
\newtheorem{Property}[lem]{Property}
\newtheorem{Properties}[lem]{Properties}
\newtheorem{Subprops}{}[lem]
\newtheorem{Para}[lem]{}

\newtheorem{rem}[lem]{Remark}

\newtheorem*{ack*}{ACKNOWLEDGEMENTS}

\newcommand{\pf}{\noindent\begin {proof}}
\newcommand{\epf}{\end{proof}}

\newcommand{\ra}{\rightarrow}

\newcommand{\Add}{\mbox{\rm Add}}
\newcommand{\add}{\mbox{\rm add}}
\begin{document}
\begin{center}
{\Large  \bf  Complete cohomology for extriangulated categories}

\vspace{0.5cm}  Jiangsheng Hu, Dongdong Zhang\footnote{Corresponding author. \\ Jiangsheng Hu was supported by the NSF of China (Grants Nos. 11671069, 11771212), Qing Lan Project of Jiangsu Province and Jiangsu Government Scholarship for Overseas Studies (JS-2019-328). Tiwei Zhao was supported by the NSF of China (Grants Nos. 11971225, 11901341) and the project ZR2019QA015 supported by Shandong Provincial Natural Science Foundation. Panyue Zhou was supported by the National Natural Science Foundation of China (Grant Nos. 11901190, 11671221),  the Hunan Provincial Natural Science Foundation of China (Grant No. 2018JJ3205) and  the Scientific Research Fund of Hunan Provincial Education Department (Grant No. 19B239).}, Tiwei Zhao and Panyue Zhou
\end{center}
\medskip
\medskip
\bigskip
\centerline { \bf  Abstract}
\medskip
\leftskip10truemm \rightskip10truemm \noindent
\hspace{1em}Let $(\mathcal{C},\mathbb{E},\mathfrak{s})$ be an extriangulated category with a proper class $\xi$ of
$\mathbb{E}$-triangles. In this paper, we study complete cohomology of objects in  $(\mathcal{C},\mathbb{E},\mathfrak{s})$ by applying $\xi$-projective resolutions and $\xi$-injective coresolutions constructed in  $(\mathcal{C},\mathbb{E},\mathfrak{s})$.  Vanishing of complete cohomology detects objects with finite $\xi$-projective dimension and  finite $\xi$-injective dimension.
As a consequence, we obtain some criteria for the validity of the Wakamatsu Tilting Conjecture and give a necessary and sufficient condition for
a virtually Gorenstein algebra to be Gorenstein. Moreover, we give a general technique for computing complete cohomology of objects with finite $\xi$-$\mathcal{G}$projective dimension. As an application, the relationships between $\xi$-projective dimensions and $\xi$-$\mathcal{G}$projective dimensions for objects in $(\mathcal{C},\mathbb{E},\mathfrak{s})$ are given.\\[2mm]

\leftskip0truemm \rightskip0truemm
%\bigskip
\section { \bf Introduction}
%\bigskip
Tate cohomology was established in the 1950s, based on Tate's observation that the $\mathbb{Z}$G-module $\mathbb{Z}$ with the trivial action admits a complete projective resolution \cite{CE}. It was further extended by Avramov and Martsinkovsky  \cite{AM} to finitely generated modules of finite Gorenstein dimension over Noetherian rings and by Veliche \cite{vg} to
complexes of finite Gorenstein projective dimension. On the other hand, Mislin \cite{Mislin}, Benson and
Carlson \cite{BF} and Vogel (first published account in  \cite{Goichot}), independently, generalized the theory to arbitrary groups. It was shown that these theories are isomorphic, and complete cohomology is a common name for them. The theory has also been extended to the setting of unbounded complexes by Asadollahi and Salarian \cite{ASCT}.

Beligiannis developed in \cite{Bel1} a relative version of homological algebra in triangulated categories in analogy to relative homological algebra in abelian categories, in which the notion of a proper class of exact sequences is replaced by a proper class of triangles. By specifying a class of triangles $\xi$, which is called a proper class of triangles, he introduced $\xi$-projective and $\xi$-injective objects. In an attempt to extend the theory, Asadollahi and Salarian \cite{AS1} introduced and studied $\xi$-Gorenstein projective, injective objects, and corresponding $\xi$-Gorenstein
dimensions in triangulated categories by modifying what Enochs, Jenda \cite{EJ2}
and Holm \cite{hG} had done in the category of modules. Motivated by the properties of Tate cohomology in the category of groups, Asadollahi and Salarian \cite{AS2} developed a Tate cohomology theory in a triangulated category $\mathcal{C}$ with a proper class of triangles, which was done for objects of finite $\xi$-Gorenstein projective dimension. They showed that
this theory not only shares basic properties with ordinary cohomology, but also enjoys
some distinctive features. We refer to \cite{RL2,RL3} for a more discussion on this matter.

The notion of extriangulated categories was introduced by Nakaoka and Palu in \cite{NP} as a simultaneous generalization of
exact categories and triangulated categories. Exact categories and extension closed subcategories of an
extriangulated category are extriangulated categories, while there exist some other examples of extriangulated categories which are neither exact nor triangulated, see \cite{HZZ,NP,ZZ}. Hence many results  on exact categories
and triangulated categories can be unified in the same framework.

Let $(\mathcal{C},\mathbb{E},\mathfrak{s})$  be an extriangulated category with a proper class $\xi$ of $\mathbb{E}$-triangles. The authors \cite{HZZ} studied a relative homological algebra in $(\mathcal{C},\mathbb{E},\mathfrak{s})$ which parallels the relative homological algebra in a triangulated category. By specifying a class of $\mathbb{E}$-triangles, which is called a proper class $\xi$ of
$\mathbb{E}$-triangles, the authors introduced $\xi$-projective dimensions and $\xi$-$\mathcal{G}$projective dimensions,
and discussed their properties.

In this paper we attempt to develop a complete cohomology theory in an extriangulated category  $(\mathcal{C},\mathbb{E},\mathfrak{s})$ and demonstrate that this theory shares some basic properties of complete cohomology groups  in the category of modules category and Tate cohomology groups in the triangulated category.

We now outline the results of the paper. In Section 2, we summarize some preliminaries
and basic facts about extriangulated categories which will be used throughout the paper.

In Section 3, for a given extriangulated category $(\mathcal{C}, \mathbb{E}, \mathfrak{s})$ with a proper class $\xi$ of $\mathbb{E}$-triangles, we define $\xi$-complete cohomology groups based on $\xi$-projective resolutions and $\xi$-injective coresolutions of objects in $(\mathcal{C}, \mathbb{E}, \mathfrak{s})$ (see Definition \ref{df:3.6}). It is proved that vanishing of $\xi$-complete cohomology characterizes the finiteness of $\xi$-projective dimension and $\xi$-injective dimension of objects in $(\mathcal{C}, \mathbb{E}, \mathfrak{s})$ (see Theorem \ref{thm1}). As consequences, we give some criteria for the validity of the Wakamatsu Tilting Conjecture; moreover, we give a necessary and sufficient condition for a virtually Gorenstein algebra to be Gorenstein (see Corollaries \ref{cor2} and \ref{cor3}).

In Section 4, we show first that an object $M$ in  an extriangulated category $(\mathcal{C}, \mathbb{E}, \mathfrak{s})$  has finite $\xi$-$\mathcal{G}$projective dimension if and only if $M$ admits a (split) $\xi$-complete resolution (see Proposition \ref{prop:3.15}). This lets us give a general technique for computing $\xi$-complete cohomology
of objects with finite $\xi$-$\mathcal{G}$projective dimension (see Theorem \ref{thm2}). As an application,
 the relationships between $\xi$-projective dimensions and
$\xi$-$\mathcal{G}$projective dimensions for objects in $(\mathcal{C},\mathbb{E},\mathfrak{s})$ are given (see Corollary \ref{cor:3.3}).

\section{\bf Preliminaries}
Throughout this paper, we always assume that  $\mathcal{C}=(\mathcal{C}, \mathbb{E}, \mathfrak{s})$ is an extriangulated category and $\xi$ is a proper class of $\mathbb{E}$-triangles in  $\mathcal{C}$.  We also assume that the extriangulated category $\mathcal{C}$ has enough $\xi$-projectives and enough $\xi$-injectives satisfying Condition (WIC). Next we briefly recall some definitions and basic properties of extriangulated categories from \cite{NP}.
We omit some details here, but the reader can find them in \cite{NP}.

Let $\mathcal{C}$ be an additive category equipped with an additive bifunctor
$$\mathbb{E}: \mathcal{C}^{\rm op}\times \mathcal{C}\rightarrow {\rm Ab},$$
where ${\rm Ab}$ is the category of abelian groups. For any objects $A, C\in\mathcal{C}$, an element $\delta\in \mathbb{E}(C,A)$ is called an $\mathbb{E}$-{\it extension}.
Let $\mathfrak{s}$ be a correspondence which associates an equivalence class $$\mathfrak{s}(\delta)=\xymatrix@C=0.8cm{[A\ar[r]^x
 &B\ar[r]^y&C]}$$ to any $\mathbb{E}$-extension $\delta\in\mathbb{E}(C, A)$. This $\mathfrak{s}$ is called a {\it realization} of $\mathbb{E}$, if it makes the diagrams in \cite[Definition 2.9]{NP} commutative.
 A triplet $(\mathcal{C}, \mathbb{E}, \mathfrak{s})$ is called an {\it extriangulated category} if it satisfies the following conditions.
\begin{enumerate}
\item $\mathbb{E}\colon\mathcal{C}^{\rm op}\times \mathcal{C}\rightarrow \rm{Ab}$ is an additive bifunctor.

\item $\mathfrak{s}$ is an additive realization of $\mathbb{E}$.

\item $\mathbb{E}$ and $\mathfrak{s}$  satisfy the compatibility conditions in \cite[Definition 2.12]{NP}.

 \end{enumerate}

\begin{rem}
Note that both exact categories and triangulated categories are extriangulated categories $($see \cite[Example 2.13]{NP}$)$ and extension closed subcategories of extriangulated categories are
again extriangulated $($see \cite[Remark 2.18]{NP}$)$. Moreover, there exist extriangulated categories which
are neither exact categories nor triangulated categories $($see \cite[Proposition 3.30]{NP}, \cite[Example 4.14]{ZZ} and \cite[Remark 3.3]{HZZ}$)$.
\end{rem}

We will use the following terminology.

\begin{definition}{ \emph{(see \cite[Definitions 2.15 and 2.19]{NP})}} {\rm
 Let $(\mathcal{C}, \mathbb{E}, \mathfrak{s})$ be an extriangulated category.
\begin{enumerate}
\item A sequence $\xymatrix@C=1cm{A\ar[r]^x&B\ar[r]^{y}&C}$ is called a {\it conflation} if it realizes some $\mathbb{E}$-extension $\delta\in\mathbb{E}(C, A)$.
In this case, $x$ is called an {\it inflation} and $y$ is called a {\it deflation}.

\item  If a conflation $\xymatrix@C=0.6cm{A\ar[r]^x&B\ar[r]^{y}&C}$ realizes $\delta\in\mathbb{E}(C, A)$, we call the pair
$\xymatrix@C=0.6cm{(A\ar[r]^x&B\ar[r]^{y}&C, \delta)}$ an {\it $\mathbb{E}$-triangle}, and write it in the following.
\begin{center} $\xymatrix{A\ar[r]^x&B\ar[r]^{y}&C\ar@{-->}[r]^{\delta}&}$\end{center}
We usually do not write this ``$\delta$" if it is not used in the argument.

\item Let $\xymatrix{A\ar[r]^x&B\ar[r]^{y}&C\ar@{-->}[r]^{\delta}&}$ and $\xymatrix{A'\ar[r]^{x'}&B'\ar[r]^{y'}&C'\ar@{-->}[r]^{\delta'}&}$
be any pair of $\mathbb{E}$-triangles. If a triplet $(a, b, c)$ realizes $(a, c): \delta\rightarrow \delta'$, then we write it as
 $$\xymatrix{A\ar[r]^{x}\ar[d]_{a}&B\ar[r]^{y}\ar[d]_{b}&C\ar[d]_{c}\ar@{-->}[r]^{\delta}&\\
 A'\ar[r]^{x'}&B'\ar[r]^{y'}&C'\ar@{-->}[r]^{\delta'}&}$$
 and call $(a, b, c)$ a {\it morphism} of $\mathbb{E}$-triangles.
\end{enumerate}}

\end{definition}

The following condition is analogous to the weak idempotent completeness in exact category (see \cite[Condition 5.8]{NP}).

\begin{cond} \label{cond:4.11} \emph{({\rm Condition (WIC)})}  Consider the following conditions.

\begin{enumerate}
\item  Let $f\in\mathcal{C}(A, B), g\in\mathcal{C}(B, C)$ be any composable pair of morphisms. If $gf$ is an inflation, then so is $f$.

\item Let $f\in\mathcal{C}(A, B), g\in\mathcal{C}(B, C)$ be any composable pair of morphisms. If $gf$ is a deflation, then so is $g$.

\end{enumerate}

\end{cond}

\begin{Ex}\label{Ex:4.12}

\emph{(1)} If $\mathcal{C}$ is an exact category, then Condition \emph{(WIC)} is equivalent to $\mathcal{C}$ is
weakly idempotent complete \emph{(see \cite[Proposition 7.6]{B"u})}.

\emph{(2)} If $\mathcal{C}$ is a triangulated category, then Condition \emph{(WIC)} is automatically satisfied.
\end{Ex}

\begin{lem}\label{lem1} \emph{(see \cite[Proposition 3.15]{NP})} Assume that $(\mathcal{C}, \mathbb{E},\mathfrak{s})$ is an extriangulated category. Let $C$ be any object, and let $\xymatrix@C=2em{A_1\ar[r]^{x_1}&B_1\ar[r]^{y_1}&C\ar@{-->}[r]^{\delta_1}&}$ and $\xymatrix@C=2em{A_2\ar[r]^{x_2}&B_2\ar[r]^{y_2}&C\ar@{-->}[r]^{\delta_2}&}$ be any pair of $\mathbb{E}$-triangles. Then there is a commutative diagram
in $\mathcal{C}$
$$\xymatrix{
    & A_2\ar[d]_{m_2} \ar@{=}[r] & A_2 \ar[d]^{x_2} \\
  A_1 \ar@{=}[d] \ar[r]^{m_1} & M \ar[d]_{e_2} \ar[r]^{e_1} & B_2\ar[d]^{y_2} \\
  A_1 \ar[r]^{x_1} & B_1\ar[r]^{y_1} & C   }
  $$
  which satisfies $\mathfrak{s}(y^*_2\delta_1)=\xymatrix@C=2em{[A_1\ar[r]^{m_1}&M\ar[r]^{e_1}&B_2]}$ and
  $\mathfrak{s}(y^*_1\delta_2)=\xymatrix@C=2em{[A_2\ar[r]^{m_2}&M\ar[r]^{e_2}&B_1].}$
%  $$m_{1*}\delta_1+m_{2*}\delta_2=0$$

\end{lem}

The following definitions are quoted verbatim from \cite[Section 3]{HZZ}. A class of $\mathbb{E}$-triangles $\xi$ is {\it closed under base change} if for any $\mathbb{E}$-triangle $$\xymatrix@C=2em{A\ar[r]^x&B\ar[r]^y&C\ar@{-->}[r]^{\delta}&\in\xi}$$ and any morphism $c\colon C' \to C$, then any $\mathbb{E}$-triangle  $\xymatrix@C=2em{A\ar[r]^{x'}&B'\ar[r]^{y'}&C'\ar@{-->}[r]^{c^*\delta}&}$ belongs to $\xi$.

Dually, a class of  $\mathbb{E}$-triangles $\xi$ is {\it closed under cobase change} if for any $\mathbb{E}$-triangle $$\xymatrix@C=2em{A\ar[r]^x&B\ar[r]^y&C\ar@{-->}[r]^{\delta}&\in\xi}$$ and any morphism $a\colon A \to A'$, then any $\mathbb{E}$-triangle  $\xymatrix@C=2em{A'\ar[r]^{x'}&B'\ar[r]^{y'}&C\ar@{-->}[r]^{a_*\delta}&}$ belongs to $\xi$.

A class of $\mathbb{E}$-triangles $\xi$ is called {\it saturated} if in the situation of Lemma \ref{lem1}, whenever  \\
$\xymatrix@C=2em{A_2\ar[r]^{x_2}&B_2\ar[r]^{y_2}&C\ar@{-->}[r]^{\delta_2 }&}$
 and $\xymatrix@C=2em{A_1\ar[r]^{m_1}&M\ar[r]^{e_1}&B_2\ar@{-->}[r]^{y_2^{\ast}\delta_1}&}$
 belong to $\xi$, then the  $\mathbb{E}$-triangle $$\xymatrix@C=2em{A_1\ar[r]^{x_1}&B_1\ar[r]^{y_1}&C\ar@{-->}[r]^{\delta_1 }&}$$  belongs to $\xi$.

An $\mathbb{E}$-triangle $\xymatrix@C=2em{A\ar[r]^x&B\ar[r]^y&C\ar@{-->}[r]^{\delta}&}$ is called {\it split} if $\delta=0$. It is easy to see that it is split if and only if $x$ is section or $y$ is retraction. The full subcategory  consisting of the split $\mathbb{E}$-triangles will be denoted by $\Delta_0$.

  \begin{definition} \emph{(see \cite[Definition 3.1]{HZZ})}\label{def:proper class} {\rm  Let $\xi$ be a class of $\mathbb{E}$-triangles which is closed under isomorphisms. Then $\xi$ is called a {\it proper class} of $\mathbb{E}$-triangles if the following conditions hold:

  \begin{enumerate}
\item  $\xi$ is closed under finite coproducts and $\Delta_0\subseteq \xi$.

\item $\xi$ is closed under base change and cobase change.

\item $\xi$ is saturated.

  \end{enumerate}}
  \end{definition}

 \begin{definition} \emph{(see \cite[Definition 4.1]{HZZ})}
 {\rm An object $P\in\mathcal{C}$  is called {\it $\xi$-projective}  if for any $\mathbb{E}$-triangle $$\xymatrix{A\ar[r]^x& B\ar[r]^y& C \ar@{-->}[r]^{\delta}& }$$ in $\xi$, the induced sequence of abelian groups $\xymatrix@C=0.6cm{0\ar[r]& \mathcal{C}(P,A)\ar[r]& \mathcal{C}(P,B)\ar[r]&\mathcal{C}(P,C)\ar[r]& 0}$ is exact. Dually, we have the definition of {\it $\xi$-injective}.}
\end{definition}

We denote by $\mathcal{P(\xi)}$ (resp., $\mathcal{I(\xi)}$) the class of $\xi$-projective (resp., $\xi$-injective) objects of $\mathcal{C}$. It follows from the definition that the subcategories $\mathcal{P}(\xi)$ and $\mathcal{I}(\xi)$ are full, additive, closed under isomorphisms and direct summands.

 An extriangulated  category $(\mathcal{C}, \mathbb{E}, \mathfrak{s})$ is said to  have {\it  enough
$\xi$-projectives} \ (resp., {\it  enough $\xi$-injectives}) provided that for each object $A$ there exists an $\mathbb{E}$-triangle $\xymatrix@C=2.1em{K\ar[r]& P\ar[r]&A\ar@{-->}[r]& }$ (resp., $\xymatrix@C=2em{A\ar[r]& I\ar[r]& K\ar@{-->}[r]&}$) in $\xi$ with $P\in\mathcal{P}(\xi)$ (resp., $I\in\mathcal{I}(\xi)$).

The {\it $\xi$-projective dimension} $\xi$-${\rm pd} A$ of $A\in\mathcal{C}$ is defined inductively.
 If $A\in\mathcal{P}(\xi)$, then define $\xi$-${\rm pd} A=0$.
Next if $\xi$-${\rm pd} A>0$, define $\xi$-${\rm pd} A\leq n$ if there exists an $\mathbb{E}$-triangle
 $K\to P\to A\dashrightarrow$  in $\xi$ with $P\in \mathcal{P}(\xi)$ and $\xi$-${\rm pd} K\leq n-1$.
Finally we define $\xi$-${\rm pd} A=n$ if $\xi$-${\rm pd} A\leq n$ and $\xi$-${\rm pd} A\nleq n-1$. Of course we set $\xi$-${\rm pd} A=\infty$, if $\xi$-${\rm pd} A\neq n$ for all $n\geq 0$.

Dually we can define the {\it $\xi$-injective dimension}  $\xi$-${\rm id} A$ of an object $A\in\mathcal{C}$.

\begin{definition} \emph{(see \cite[Definition 4.4]{HZZ})}
{\rm A {\it $\xi$-exact} complex $\mathbf{X}$ is a diagram $$\xymatrix@C=2em{\cdots\ar[r]&X_1\ar[r]^{d_1}&X_0\ar[r]^{d_0}&X_{-1}\ar[r]&\cdots}$$ in $\mathcal{C}$ such that for each integer $n$, there exists an $\mathbb{E}$-triangle $\xymatrix@C=2em{K_{n+1}\ar[r]^{g_n}&X_n\ar[r]^{f_n}&K_n\ar@{-->}[r]^{\delta_n}&}$ in $\xi$ and $d_n=g_{n-1}f_n$.
}\end{definition}

\begin{definition} \emph{(see \cite[Definition 4.5]{HZZ})}
{\rm Let $\mathcal{W}$ be a class of objects in $\mathcal{C}$. An $\mathbb{E}$-triangle
$$\xymatrix@C=2em{A\ar[r]& B\ar[r]& C\ar@{-->}[r]& }$$ in $\xi$ is called to be
{\it $\mathcal{C}(-,\mathcal{W})$-exact} (resp.,
{\it $\mathcal{C}(\mathcal{W},-)$-exact}) if for any $W\in\mathcal{W}$, the induced sequence of abelian groups $\xymatrix@C=2em{0\ar[r]&\mathcal{C}(C,W)\ar[r]&\mathcal{C}(B,W)\ar[r]&\mathcal{C}(A,W)\ar[r]& 0}$ (resp., \\ $\xymatrix@C=2em{0\ar[r]&\mathcal{C}(W,A)\ar[r]&\mathcal{C}(W,B)\ar[r]&\mathcal{C}(W,C)\ar[r]&0}$) is exact in ${\rm Ab}$}.
\end{definition}

\begin{definition} \emph{(see \cite[Definition 4.6]{HZZ})}
 {\rm Let $\mathcal{W}$ be a class of objects in $\mathcal{C}$. A complex $\mathbf{X}$ is called {\it $\mathcal{C}(-,\mathcal{W})$-exact} (resp.,
{\it $\mathcal{C}(\mathcal{W},-)$-exact}) if it is a $\xi$-exact complex
$$\xymatrix@C=2em{\cdots\ar[r]&X_1\ar[r]^{d_1}&X_0\ar[r]^{d_0}&X_{-1}\ar[r]&\cdots}$$ in $\mathcal{C}$ such that  there is a $\mathcal{C}(-,\mathcal{W})$-exact (resp.,
 $\mathcal{C}(\mathcal{W},-)$-exact) $\mathbb{E}$-triangle $$\xymatrix@C=2em{K_{n+1}\ar[r]^{g_n}&X_n\ar[r]^{f_n}&K_n\ar@{-->}[r]^{\delta_n}&}$$ in $\xi$ for each integer $n$ and $d_n=g_{n-1}f_n$.

 A $\xi$-exact complex $\mathbf{X}$ is called {\it complete $\mathcal{P}(\xi)$-exact} (resp., {\it complete $\mathcal{I}(\xi)$-exact}) if it is $\mathcal{C}(-,\mathcal{P}(\xi))$-exact (resp.,
 $\mathcal{C}(\mathcal{I}(\xi),-)$-exact).}
\end{definition}

\begin{definition} \emph{(see \cite[Definition 4.7]{HZZ})}
 {\rm A  {\it complete $\xi$-projective resolution}  is a complete $\mathcal{P}(\xi)$-exact complex\\ $$\xymatrix@C=2em{\mathbf{P}:\cdots\ar[r]&P_1\ar[r]^{d_1}&P_0\ar[r]^{d_0}&P_{-1}\ar[r]&\cdots}$$ in $\mathcal{C}$ such that $P_n$ is $\xi$-projective for each integer $n$. Dually,   a  {\it complete $\xi$-injective coresolution}  is a complete $\mathcal{I}(\xi)$-exact complex $$\xymatrix@C=2em{\mathbf{I}:\cdots\ar[r]&I_1\ar[r]^{d_1}&I_0\ar[r]^{d_0}&I_{-1}\ar[r]&\cdots}$$ in $\mathcal{C}$ such that $I_n$ is $\xi$-injective for each integer $n$.}
\end{definition}

\begin{definition} \emph{(see \cite[Definition 4.8]{HZZ})}
{\rm  Let $\mathbf{P}$ be a complete $\xi$-projective resolution in $\mathcal{C}$. So for each integer $n$, there exists a $\mathcal{C}(-, \mathcal{P}(\xi))$-exact $\mathbb{E}$-triangle $\xymatrix@C=2em{K_{n+1}\ar[r]^{g_n}&P_n\ar[r]^{f_n}&K_n\ar@{-->}[r]^{\delta_n}&}$ in $\xi$. The objects $K_n$ are called {\it $\xi$-$\mathcal{G}$projective} for each integer $n$. Dually if  $\mathbf{I}$ is a complete $\xi$-injective  coresolution in $\mathcal{C}$, there exists a  $\mathcal{C}(\mathcal{I}(\xi), -)$-exact $\mathbb{E}$-triangle $\xymatrix@C=2em{K_{n+1}\ar[r]^{g_n}&I_n\ar[r]^{f_n}&K_n\ar@{-->}[r]^{\delta_n}&}$ in $\xi$ for each integer $n$. The objects $K_n$ are called {\it $\xi$-$\mathcal{G}$injective} for each integer $n$.}
\end{definition}

%Similar to the way of defining $\xi$-projective and $\xi$-injective dimensions, for an object $A\in{\mathcal{C}}$, the $\xi$-$\mathcal{G}$projective dimension $\xi$-${\rm \mathcal{G}pd} A$ and $\xi$-$\mathcal{G}$injective dimension $\xi$-${\rm \mathcal{G}id} A$ are defined inductively in \cite{HZZ}.

 We denote by $\mathcal{GP}(\xi)$ (resp., $\mathcal{GI}(\xi)$) the class of $\xi$-$\mathcal{G}$projective (resp., $\xi$-$\mathcal{G}$injective) objects.
It is obvious that $\mathcal{P(\xi)}$ $\subseteq$ $\mathcal{GP}(\xi)$ and $\mathcal{I(\xi)}$ $\subseteq$ $\mathcal{GI}(\xi)$.

\section{\bf $\xi$-complete cohomology}

In this section, we study $\xi$-complete cohomology groups of objects in an extriangulated category $\mathcal{C}$. At first, we need the following definition.

\begin{definition} \emph{(see \cite[Definition 3.1]{HZZ1})}\label{df:resolution} {\rm Let $M$ be an object in $\mathcal{C}$. A {\it $\xi$-projective resolution} of $M$ is a $\xi$-exact complex $\mathbf{P}\rightarrow M$ such that $\mathbf{P}_n\in{\mathcal{P}(\xi)}$ for all $n\geq0$. Dually, a {\it $\xi$-injective coresolution} of $M$ is a $\xi$-exact complex $ M\rightarrow \mathbf{I}$ such that $\mathbf{I}_n\in{\mathcal{I}(\xi)}$ for all $n\leq0$.}
\end{definition}

Using standard arguments
from relative homological algebra, one can prove the following version of the comparison
theorem for $\xi$-projective resolutions (resp., $\xi$-injective coresolutions).

\begin{lem}\label{lemma:comparasion-lemma} Let $M$ and $N$ be objects in $\mathcal{C}$.
\begin{enumerate}
\item
If $f_{M}:\xymatrix@C=2em{\mathbf{P}_M\ar[r]& M}$ and $f_{N}:\xymatrix@C=2em{\mathbf{P}_N\ar[r]& N}$ are $\xi$-projective resolutions of $M$ and $N$, respectively, then for each morphism $\mu:M\rightarrow N$, there exists a morphism $\widetilde{\mu}$, unique up to homotopy, making the following diagram
$$\xymatrix{\mathbf{P}_M\ar[r]^{f_{M}}\ar[d]_{\widetilde{\mu}}&M\ar[d]^{\mu}\\
\mathbf{P}_N\ar[r]^{f_{N}}&N}$$ \noindent
commute. If $\mu$ is an isomorphism, then $\widetilde {\mu}$ is a homotopy equivalence.

\item If $g_{M}:\xymatrix@C=2em{M\ar[r]& \mathbf{I}_{M}}$ and $g_{N}:\xymatrix@C=2em{N\ar[r]& \mathbf{I}_N}$ are $\xi$-injective coresolutions of $M$ and $N$, respectively, then for each morphism $\nu:M\rightarrow N$, there exists a morphism $\widetilde{\nu}$, unique up to homotopy, making the following diagram
$$\xymatrix{M\ar[r]^{g_{M}}\ar[d]_{\nu}&\mathbf{I}_{M}\ar[d]^{\widetilde{\nu}}\\
N\ar[r]^{g_{N}}&\mathbf{I}_{N}}$$ \noindent
commute. If $\nu$ is an isomorphism, then $\widetilde {\nu}$ is a  homotopy equivalence.
\end{enumerate}
\end{lem}

We denote by $\textrm{Ch}(\mathcal{C})$ the category of complexes in $\mathcal{C}$; the objects are complexes and morphisms are chain maps. We write the complexes homologically, so an object $\mathbf{X}$ of $\textrm{Ch}(\mathcal{C})$ is of the form
$$\xymatrix@C=2em{\mathbf{X}:=\cdots \ar[r]&X_{n+1}\ar[r]^{d_{n+1}^{\mathbf{X}}}&X_n\ar[r]^{d_n^{\mathbf{X}}}&X_{n-1}\ar[r]&\cdots}.$$
Assume that $\mathbf{X}$ and $\mathbf{Y}$ are complexes in $\textrm{Ch}(\mathcal{C})$.
A homomorphism $\xymatrix@C=2em{\varphi:\mathbf{X}\ar[r]&\mathbf{Y}}$ of degree $n$ is a family $(\varphi_i)_{i\in\mathbb{Z}}$ of morphisms $\xymatrix@C=2em{\varphi_i:X_i\ar[r]& Y_{i+n}}$ for all $i\in\mathbb{Z}$. In this case, we set $|\varphi|=n$. All such homomorphisms form an abelian group, denoted by $\mathcal{C}(\mathbf{X},\mathbf{Y})_n$, which is identified with $\prod_{i\in \mathbb{Z}}{\rm \mathcal{C}}(X_i,Y_{i+n})$. We let $\mathcal{C}(\mathbf{X},\mathbf{Y})$ be the complex of abelian groups with $n$th component $\mathcal{C}(\mathbf{X},\mathbf{Y})_n$ and differential $d(\varphi_i)=d_{i+n}^{\mathbf{Y}}\varphi_i-(-1)^n\varphi_{i-1}d_i^{\mathbf{X}}$ for $\varphi=(\varphi_i)\in\mathcal{C}(\mathbf{X},\mathbf{Y})_n$.
We refer to \cite{AFH} for more details.

\begin{rem} \label{3.3} Let $M$ and $N$ be objects in $\mathcal{C}$.
\begin{enumerate}
\item Note that there are two $\xi$-projective resolutions $\xymatrix@C=2em{\mathbf{P}_M\ar[r]& M}$ and $\xymatrix@C=2em{\mathbf{P}_N\ar[r]& N}$ of $M$ and $N$, respectively. A homomorphism $\beta\in \mathcal{C}(\mathbf{P}_M,\mathbf{P}_N)$ is {\it bounded above} if $\beta_i=0$ for all $i\gg 0$. The subset $\overline{\mathcal{C}}(\mathbf{P}_M,\mathbf{P}_N)$, consisting of all bounded above homomorphisms, is a subcomplex with components
\begin{center}$\overline{\mathcal{C}}(\mathbf{P}_M,\mathbf{P}_N)_n=\{(\varphi_i)\in \mathcal{C}(\mathbf{P}_M,\mathbf{P}_N)_n \ | \ \varphi_i=0$ for all $i\gg 0\}.$\end{center}
We set
$$\widetilde{\mathcal{C}}(\mathbf{P}_M,\mathbf{P}_N)={\mathcal{C}}(\mathbf{P}_M,\mathbf{P}_N)/\overline{\mathcal{C}}(\mathbf{P}_M,\mathbf{P}_N).$$

\item Note that there are two $\xi$-injective coresolutions $\xymatrix@C=2em{M\ar[r]&\mathbf{I}_M}$ and $\xymatrix{N\ar[r]&\mathbf{I}_N}$ of $M$ and $N$, respectively. A homomorphism $\beta\in {\mathcal{C}}(\mathbf{I}_M,\mathbf{I}_N)$ is {\it bounded below} if $\beta_i=0$ for all $i\ll 0$. The subset $\underline{\mathcal{C}}(\mathbf{I}_M,\mathbf{I}_N)$, consisting of all bounded below homomorphisms, is a subcomplex with components
\begin{center}$\underline{\mathcal{C}}(\mathbf{I}_M,\mathbf{I}_N)_n=\{(\varphi_i)\in {\mathcal{C}}(\mathbf{I}_M,\mathbf{I}_N)_n \ | \ \varphi_i=0$ for all $i\ll 0\}.$\end{center}
We set
$$\widetilde{\mathcal{C}}(\mathbf{I}_M,\mathbf{I}_N)={\mathcal{C}}(\mathbf{I}_M,\mathbf{I}_N)/\underline{\mathcal{C}}(\mathbf{I}_M,\mathbf{I}_N).$$
\end{enumerate}
\end{rem}

\begin{definition}\label{df:3.6} {\rm Let $M$ and $N$ be objects in $\mathcal{C}$, and let $n$ be an integer.
\begin{enumerate}
\item Using $\xi$-projective resolutions, we define the $n$th \emph{$\xi$-complete cohomology group}, denoted by $\widetilde{\rm \xi xt}_{\mathcal{P}}^n(M,N)$, as
$$\widetilde{\rm \xi xt}_{\mathcal{P}}^n(M,N)=H^n(\widetilde{\mathcal{C}}(\mathbf{P}_M,\mathbf{P}_N)),$$
where $\widetilde{\mathcal{C}}(\mathbf{P}_M,\mathbf{P}_N)$ is the complex defined in Remark \ref{3.3}(1).

\item Using $\xi$-injective coresolutions, we define the $n$th \emph{$\xi$-complete cohomology group}, denoted by $\widetilde{\rm \xi xt}_{\mathcal{I}}^n(M,N)$, as
$$\widetilde{\rm \xi xt}_{\mathcal{I}}^n(M,N)=H^n(\widetilde{\mathcal{C}}(\mathbf{I}_M,\mathbf{I}_N)),$$
where $\widetilde{\mathcal{C}}(\mathbf{I}_M,\mathbf{I}_N)$ is the complex defined in Remark \ref{3.3}(2).
\end{enumerate}}
\end{definition}

\begin{rem}\label{fact:2.5'} {By Lemma \ref{lemma:comparasion-lemma}, one can see that $\widetilde{\rm \xi xt}_{\mathcal{P}}^n(-,-)$ and $\widetilde{\rm \xi xt}_{\mathcal{I}}^n(-,-)$ are cohomological functors for any integer $n\in\mathbb{Z}$, independent of the choice of $\xi$-projective resolutions and $\xi$-injective coresolutions, respectively.}
\end{rem}

\begin{definition} \emph{(see \cite[Definition 3.2]{HZZ1})}\label{df:derived-functors} {\rm Let $M$ and $N$ be objects in $\mathcal{C}$.

\begin{enumerate}
\item[{\rm (1)}] If we choose a $\xi$-projective resolution $\xymatrix@C=2em{\mathbf{P}\ar[r]& M}$ of  $M$, then for any integer $n\geq 0$, the \emph{$\xi$-cohomology groups} $\xi{\rm xt}_{\mathcal{P}(\xi)}^n(M,N)$ are defined as
$$\xi{\rm xt}_{\mathcal{P}(\xi)}^n(M,N)=H^n({\mathcal{C}}(\mathbf{P},N)).$$

\item[{\rm (2)}] If we choose a
$\xi$-injective coresolution $\xymatrix@C=2em{N\ar[r]&\mathbf{I}}$ of  $N$, then for any integer $n\geq 0$, the \emph{$\xi$-cohomology groups} $\xi{\rm xt}_{\mathcal{I}(\xi)}^n(M,N)$ are defined as $$\xi{\rm xt}_{\mathcal{I}(\xi)}^n(M,N)=H^n({\mathcal{C}}(M, \mathbf{I})).$$
\end{enumerate}}
\end{definition}

\begin{rem}\label{fact:2.5'} { By Lemma \ref{lemma:comparasion-lemma}, one can see that ${\rm \xi xt}_{\mathcal{P}(\xi)}^n(-,-)$ and ${\rm \xi xt}_{\mathcal{I}(\xi)}^n(-,-)$ are cohomological functors for any integer $n\geq 0$, independent of the choice of $\xi$-projective resolutions and $\xi$-injective coresolutions, respectively. In fact, with the modifications of the usual proof, one obtains the isomorphism $\xi{\rm xt}_{\mathcal{P}(\xi)}^n(M,N)\cong \xi{\rm xt}_{\mathcal{I}(\xi)}^n(M,N),$
which is denoted by $\xi{\rm xt}_{\xi}^n(M,N).$
}
\end{rem}
\begin{lem} \label{lem:3.8}Let $M$ and $N$ be objects in $\mathcal{C}$.
\begin{enumerate}
\item If  $\xymatrix@C=2em{f:\mathbf{P}\ar[r]& M}$ and $\xymatrix@C=2em{g:\mathbf{Q}\ar[r]& N}$ are $\xi$-projective resolutions of $M$ and $N$, respectively, then we have $${\rm \xi xt}_{\xi}^n(M,N)=H^n({\mathcal{C}}(\mathbf{P},\mathbf{Q}))$$
for each integer $n\geq1$.

\item If  $\xymatrix@C=2em{f':M\ar[r]& \mathbf{I}}$ and $\xymatrix@C=2em{g':N\ar[r]& \mathbf{J}}$ are $\xi$-injective coresolutions of $M$ and $N$, respectively, then we have $${\rm \xi xt}_{\xi}^n(M,N)=H^n({\mathcal{C}}(\mathbf{I},\mathbf{J}))$$
for each integer $n\geq1$.
\end{enumerate}

\end{lem}
\begin{proof} We only prove (1) and the proof of (2) is similar. Since $\xymatrix@C=2em{g:\mathbf{Q}\ar[r]& N}$ is a $\xi$-projective resolution $N$, we have that the mapping cone $$\textrm{Cone}(g):= \cdots\to Q_1\to Q_0\to N\to 0$$ is a $\xi$-exact complex. Note that
$\mathbf{P}:= \cdots \ra P_1\ra P_0\ra 0$ is a complex such that $P_i$ is $\xi$-projective for any $i\geq0$. Then $\mathcal{C}(\mathbf{P},{\rm Cone}(g))$ is an exact complex and the proof is similar to \cite[Lemma 2.4]{cfh}. Since ${\mathcal{C}}(\mathbf{P},{\rm Cone}(g))\cong {\rm Cone}({\mathcal{C}}(\mathbf{P},g))$, we get that ${\rm Cone}(\mathcal{C}(\mathbf{P},g))$ is an exact complex. So ${\rm \xi xt}_{\xi}^n(M,N)=H^n({\mathcal{C}}(\mathbf{P},N))\cong H^n({\mathcal{C}}(\mathbf{P},\mathbf{Q}))$ for each integer $n\geq1$. This completes the proof.
\end{proof}

Recall from \cite[Definition 4.1]{HZZ1} that a full subcategory $\mathcal{X}\subseteq \mathcal{C}$ is called a {\it generating subcategory} of $\mathcal{C}$ if for all $M\in{\mathcal{C}}$, $\mathcal{C}(\mathcal{X},M)=0$ implies that $M=0$. Dually, a full subcategory $\mathcal{Y}\subseteq\mathcal{C}$ is called a \emph{cogenerating subcategory} of $\mathcal{C}$ if for
all $N\in{\mathcal{C}}$, $\mathcal{C}(N,\mathcal{Y})=0$ implies that $N=0$.

\begin{lem}\label{lem:3.9}  Let $M$ be an object in $\mathcal{C}$ and $n$ a non-negative integer.
\begin{enumerate}
\item If $\mathcal{P}(\xi)$ is a generating subcategory of $\mathcal{C}$, then $\xi$-${\rm pd}M\leq n$ if and only if ${\rm \xi}xt_{\xi}^{n+1}(M,N)=0$ for any object $N$ in $\mathcal{C}$.

\item If $\mathcal{I}(\xi)$ is a cogenerating subcategory of $\mathcal{C}$, then $\xi$-${\rm id}M\leq n$ if and only if ${\rm \xi}xt_{\xi}^{n+1}(N,M)=0$ for any object $N$ in $\mathcal{C}$.
\end{enumerate}
\end{lem}
\begin{proof} The proofs are similar to \cite[Proposition 4.17]{Bel1}.
\end{proof}

We are now in a position to prove the main result of this section.

\begin{thm}\label{thm1} Let $M$ be an object in $\mathcal{C}$.
\begin{enumerate}

\item If $\mathcal{P}(\xi)$ is a generating subcategory of $\mathcal{C}$, then the following are equivalent:
\begin{enumerate}
\item[(a)] $\xi$-${\rm pd}M<\infty$;

\item[(b)] $\widetilde{{\rm \xi xt}}_{\mathcal{P}}^i(M,N)=0$ for any integer $i\in\mathbb{Z}$ and any object $N$ in $\mathcal{C}$;

\item[(c)] $\widetilde{{\rm \xi xt}}_{\mathcal{P}}^i(N,M)=0$ for any integer $i\in\mathbb{Z}$ and any object $N$ in $\mathcal{C}$;

\item[(d)] $\widetilde{{\rm \xi xt}}_{\mathcal{P}}^0(M,M)=0$.

\end{enumerate}

\item If $\mathcal{I}(\xi)$ is a cogenerating subcategory of $\mathcal{C}$, then the following are equivalent:
\begin{enumerate}
\item[(a)] $\xi$-${\rm id}M<\infty$;

\item[(b)] $\widetilde{{\rm \xi xt}}_{\mathcal{I}}^i(N,M)=0$ for all integer $i\in\mathbb{Z}$ and any object $N$ in $\mathcal{C}$;

\item[(c)] $\widetilde{{\rm \xi xt}}_{\mathcal{I}}^i(M,N)=0$ for all integer $i\in\mathbb{Z}$ and any object $N$ in $\mathcal{C}$;

\item[(d)] $\widetilde{{\rm \xi xt}}_{\mathcal{I}}^0(M,M)=0$.
\end{enumerate}

\end{enumerate}
\end{thm}
\begin{proof} We prove part (1); the proof of (2) is dual. $(a)\Rightarrow (b)$. Since $\xi$-${\rm pd}M<\infty$, there is a $\xi$-projective resolution of $M$,
say $\mathbf{P}_M$, with $(\mathbf{P}_M)_i=0$ for all $i\gg 0$. Hence $\overline{\mathcal{C}}(\mathbf{P}_M,-)={\mathcal{C}}(\mathbf{P}_M,-)$,
 so $\widetilde{{\rm \xi xt}}_{\mathcal{P}}^i(M,N)=0$ for any integer $i\in\mathbb{Z}$ and any object $N$ in $\mathcal{C}$.

$(b)\Rightarrow (d)$ is trivial.

$(d)\Rightarrow (a)$. Let $\xymatrix@C=2em{\mathbf{P}_M\ar[r]& M}$ be a $\xi$-projective resolution of $M$.
Then we have $${\rm id}_{\mathbf{P}_M}+\overline{\mathcal{C}}(\mathbf{P}_M,\mathbf{P}_M)_0\in{\rm Z}_0(\widetilde{\mathcal{C}}(\mathbf{P}_M,\mathbf{P}_M)).$$
Note that $\widetilde{{\rm \xi xt}}_{\mathcal{P}}^0(M,M)=0$ by hypothesis, then we have
$${\rm id}_{\mathbf{P}_M}+\overline{\mathcal{C}}(\mathbf{P}_M,\mathbf{P}_M)_0\in{\rm B}_0(\widetilde{\mathcal{C}}(\mathbf{P}_M,\mathbf{P}_M)).$$
Hence there is a $\varphi\in{\mathcal{C}}(\mathbf{P}_M,\mathbf{P}_M)_1$ such that
$${\rm id}_{\mathbf{P}_M}-d^{{\mathcal{C}}(\mathbf{P}_M,\mathbf{P}_M)}(\varphi)\in\overline{\mathcal{C}}(\mathbf{P}_M,\mathbf{P}_M)_0.$$
Consequently, there is  $\psi\in\overline{\mathcal{C}}(\mathbf{P}_M,\mathbf{P}_M)_0$ such that
${\rm id}_{\mathbf{P}_M}-\psi=d^{{\mathcal{C}}(\mathbf{P}_M,\mathbf{P}_M)}(\varphi)$. Since $\psi$ is bounded above,
$({\rm id}_{\mathbf{P}_M}-\psi)_j$ is the identity morphism on $(\mathbf{P}_M)_j$ for all $j\gg 0$.
 Let  $m$ be an integer such that $d_{s+1}^{\mathbf{P}_M}\varphi_s+\varphi_{s-1}d_{s}^{\mathbf{P}_M}={\rm id}_{(\mathbf{P}_M)_s}$ for any $s\geq m$.
  It is easy to check that ${\rm\xi xt}_{\xi}^{m+1}(M,N)=0$ for any object $N$ in $\mathcal{C}$, which deduces  $\xi$-${\rm pd}M\leq m$ by Lemma \ref{lem:3.9}(1).

The implications of $(a)\Rightarrow (c)\Rightarrow (d)$ can be proved similarly.
\end{proof}

Assume that $(\mathcal{T},\Sigma, \Delta)$ is a compactly generated triangulated category, where $\Sigma$ is the suspension functor and $\Delta$ is the triangulation. It follows from Krause \cite{Krause} and Beligiannis \cite{Bel1} that the class $\xi$ of pure triangles
(which is induced by the compact objects) is proper and $\mathcal{T}$ has enough $\xi$-projectives or $\xi$-injectives. Moreover, the object $X\in{\mathcal{P(\xi)}}$ is called the \emph{pure-projective} (see \cite[Definition 1.2(2)]{Krause}).  Following Beligiannis \cite{Bel1}, we set $\xi$-gl.dim$\mathcal{T}$ $=$ $\sup\{\xi\textrm{-}{\rm pd}M \ | \ \textrm{for} \ \textrm{any} \ M\in{\mathcal{T}}\}$.
As a corollary of Theorem \ref{thm1}, we have the following.

 \begin{cor}\label{cor1} Let $(\mathcal{T},\Sigma, \Delta)$ be a compactly generated triangulated category. Then the following are equivalent:
\begin{enumerate}
\item $\xi$-gl.dim$\mathcal{T}$ $<\infty$;

\item $\widetilde{{\rm \xi xt}}_{\mathcal{P}}^i(M,N)=0$ for any integer $i\in\mathbb{Z}$ and all objects $M$ and $N$ in $\mathcal{T}$;

\item $\widetilde{{\rm \xi xt}}_{\mathcal{P}}^0(M,M)=0$ for any object $M$ in $\mathcal{T}$.

\end{enumerate}
\end{cor}

Let $R$ be a ring. Assume that $\mathcal{X}$ is a class of left $R$-modules. For a left $R$-module $M$, write $M\in{^{\perp_{\infty}}\mathcal{X}}$ (resp., $M\in{^{\perp}\mathcal{X}}$) if $\textrm{Ext}_{R}^{\geq1}(M,X)=0$ (resp., $\textrm{Ext}_{R}^{1}(M,X)=0$) for each $X\in\mathcal{X}$. Dually, we can define $M\in{\mathcal{X}^{\perp_{\infty}}}$ and $M\in{\mathcal{X}^{\perp}}$.
Given a left $R$-module $M$, we denote by $\Add M$ (resp., $\add M$) the class of all modules that
are isomorphic to direct summands of direct sums (resp., finite direct sums) of copies of $M$.

Recall that a left $R$-module $G$ is called \emph{tilting} \cite{LAHFUCol01,Colby1995} provided that the following hold:
\begin{enumerate}
\item[$(T1)$] $G$ has finite projective dimension.
\item[$(T2)$] $\textrm{Ext}_{R}^{i}(G,G^{(\lambda)})$ $=$ $0$ for each $i$ $\geq$ $1$ and for every cardinal $\lambda$.
\item[$(T3)$] There is a long exact sequence $0\ra R\ra G_{0}\ra \cdots\ra G_{r}\ra 0$ with $G_{i}$ $\in$ $\textrm{Add} T$ for $0$ $\leq$ $i$ $\leq$ $r$, where $r$ is the projective dimension of $G$.
\end{enumerate}

Moreover, a left $R$-module $\omega$ is said to be a \emph{Wakamatsu tilting module} \cite{Mantese,Wamastilting} if it has the following properties:

\begin{enumerate}
\item[$(W1)$] There exists an exact sequence
$\cdots\ra P_i\ra\cdots\ra  P_0\ra \omega\ra 0$
with  $P_i$ finitely generated and projective for each $i$ $\geq$ $0$;
\item [$(W2)$] $\textrm{Ext}_{R}^{i}(\omega,\omega)$ $=$ $0$ for each $i$ $\geq$ $1$;

\item[$(W3)$] There exists an exact sequence $0\ra R\ra \omega_0 \xrightarrow{f_0}  \cdots \ra \omega_i\xrightarrow{f_i}\cdots$ with $\omega_i$ $\in$ add$\omega$
and $\ker(f_i)$ $\in$ $^{\bot_\infty}\omega$ for each $i$ $\geq$ $0$.
\end{enumerate}

It is still an open problem whether a Wakamatsu tilting $R$-module of
finite projective dimension must be a tilting $R$-module whenever $R$ is an Artin algebra. This is known as Wakamatsu Tilting
Conjecture (see \cite[Chapter IV]{Beli and Reiten}). This conjecture is also related to many other homological conjectures and attract many
algebraists, see for instance \cite{Beli and Reiten,Huang,Mantese,Wang2019,Wei}.

Denote by $\mathcal{FP}_{\infty}$ the class of left $R$-modules possessing a projective resolution consisting of
finitely generated modules. Therefore, we have the following corollary which is a consequence of Theorem \ref{thm1}.

 \begin{cor}\label{cor2} Let $R$ be a ring and $\omega$ a Wakamatsu tilting left $R$-module with finite projective dimension. Fix an exact sequence $0\ra R\ra \omega_0 \xrightarrow{f_0} \cdots \ra \omega_i\xrightarrow{f_i}\cdots$ with $\omega_i$ $\in$ {\rm add$\omega$}
and $\ker (f_i)$ $\in$ $^{\bot_\infty}\omega$ for  $i$ $\geq$ $0$. If we set $A$ $=$ ${\bigoplus\limits_{i\geq0}}\ker (f_i)$, then the following are equivalent:

\begin{enumerate}
\item $\omega$ is a tilting module;

\item $A$ has finite projective dimension;

\item $\widetilde{{\rm \xi xt}}_{\mathcal{P}}^i(A,N)=0$ for any integer $i\in\mathbb{Z}$ and any left $R$-module $N$;

\item $\widetilde{{\rm \xi xt}}_{\mathcal{P}}^i(N,A)=0$ for any integer $i\in\mathbb{Z}$ and any left $R$-module $N$;
\item $\widetilde{{\rm \xi xt}}_{\mathcal{P}}^0(A,A)=0$.
\end{enumerate}
\end{cor}
 \begin{proof} $(1)\Rightarrow(2)$ holds by the proof of (1)$\Rightarrow$(2) in \cite[Theorem 4.4]{Wang2019}.

 $(2)\Rightarrow(1)$. Note that $\ker (f_i)$ $\in$ $\mathcal{FP}_{\infty}$ for any $i\geq0$ by \cite[Theorem 1.8]{BP}. It follows from \cite[Lemma 3.1.6]{GT} that $A^{\bot_\infty}$ is closed under direct sums. Since $A$ has finite projective dimension, each module $M$ in $^\bot(A^{\bot_{\infty}})$ has finite projective dimension by \cite[Lemma 2.2(2)]{Wang2019}. If we set $\mathcal{K}_{A}={^\bot(A^{\bot_{\infty}})}\cap A^{\bot_{\infty}}$, then there is a tilting $R$-module $T$ such that $\mathcal{K}_{A}=\textrm{Add}T$ by \cite[Theorem 1.1]{Wang2019}. So (1) holds by \cite[Corollary 4.5]{Wang2019}.

 $(2)\Leftrightarrow (3)\Leftrightarrow (4)\Leftrightarrow(5)$ follow from Theorem \ref{thm1}.
 \end{proof}

%\begin{rem} \label{rem: 2017081701} \rm{
%It is still an open problem whether a Wakamatsu tilting $R$-module of
%finite projective dimension must be a tilting $R$-module whenever $R$ is an Artin algebra. This is known as Wakamatsu Tilting
%Conjecture (see \cite[Chapter IV]{Beli and Reiten}). Mantese and Reiten  \cite{Mantese} showed that the Wakamatsu Tilting Conjecture is a special case of the Finitistic Dimension Conjecture. These conjectures are also related to many other homological conjectures and attract
%many algebraists, see for instance \cite{Anegeleri,Beli and Reiten,Celikbas,Mantese,Wei2011,XiCC,Zimmermann}.}
%\end{rem}

Let $\Lambda$ be an Artin algebra. Denote by $\mathcal{GP}(\Lambda)$ the class of Gorenstein projective left $\Lambda$-modules and by $\mathcal{GI}(\Lambda)$ the class of Gorenstein injective left $\Lambda$-modules. Recall from Beligiannis and Krause \cite{Beliandkrause} that $\Lambda$ is called \emph{virtually Gorenstein} if  $\mathcal{GP}(\Lambda)^{\perp}= {^{\perp}\mathcal{GI}(\Lambda)}$. It is not hard to check that for any virtually Gorenstein algebra $\Lambda$,  an exact sequence
$0\ra X\ra Y\ra Z\ra 0$ of left $\Lambda$-modules  is $\textrm{Hom}_{\Lambda}(\mathcal{GP}(\Lambda),-)$-exact if and only if it is $\textrm{Hom}_{\Lambda}(-,\mathcal{GI}(\Lambda))$-exact.  By \cite[Theorem 2.2]{WLH}, we have an exact category $(\textrm{Mod}\Lambda,\varepsilon)$, where $\textrm{Mod}\Lambda$ is the class of left $\Lambda$-modules and $\varepsilon$ is the class of $\textrm{Hom}_{\Lambda}(\mathcal{GP}(\Lambda),-)$-exact sequences $0\ra X\ra Y\ra Z\ra 0$ of left $\Lambda$-modules. We refer to \cite{chen,sergio,ZAD} for a more discussion on this matter.

On the other hand,
it follows from \cite[Theorem 12.3.1]{EJ2} that $\Lambda$ is Gorenstein if and only if every left $\Lambda$-module has finite Gorenstein projective dimension.  Note that a virtually Gorenstein algebra is not Gorenstein  in general by
\cite[4.3, p.560]{Beliandkrause}. By Theorem \ref{thm1}, we have the following corollary which characterizes when a virtually Gorenstein algebra is Gorenstein.

\begin{cor}\label{cor3} Let $\Lambda$ be a virtually Gorenstein algebra. If we assume that $\mathcal{C}$ is the exact category {\rm $(\textrm{Mod}\Lambda,\varepsilon)$} and $\xi=\varepsilon$ is the class of $\textrm{\rm Hom}_{\Lambda}(\mathcal{GP}(\Lambda),-)$-exact sequences of left $\Lambda$-modules
$0\ra X\ra Y\ra Z\ra 0$, then the following are equivalent:

\begin{enumerate}
\item $\Lambda$ is Gorenstein;

\item $\widetilde{{\rm \xi xt}}_{\mathcal{P}}^i(M,N)=0$ for any integer $i\in\mathbb{Z}$ and all left $\Lambda$-modules $M$ and $N$;

\item $\widetilde{{\rm \xi xt}}_{\mathcal{P}}^0(M,M)=0$ for any left $\Lambda$-module $M$;

\item $\widetilde{{\rm \xi xt}}_{\mathcal{I}}^i(M,N)=0$ for any integer $i\in\mathbb{Z}$ and all left $\Lambda$-modules $M$ and $N$;

\item $\widetilde{{\rm \xi xt}}_{\mathcal{I}}^0(M,M)=0$ for any left $\Lambda$-module $M$.

\end{enumerate}
\end{cor}

\section{\bf $\xi$-complete cohomology for objects with finite $\xi$-$\mathcal{G}$projective  dimension}

Recall from \cite{HZZ} that the $\xi$-$\mathcal{G}$projective dimension $\xi$-$\mathcal{G}$${\rm pd} M$ of an object $M\in\mathcal{C}$ is defined inductively. If $M\in\mathcal{GP}(\xi)$ then define $\xi$-$\mathcal{G}$${\rm pd} M=0$.
Next by induction, for an integer $n>0$, put $\xi$-$\mathcal{G}$${\rm pd} M\leq n$ if there exists an $\mathbb{E}$-triangle $\xymatrix{K\ar[r]&G\ar[r]&M\ar@{-->}[r]&}$ in $\xi$ with $G\in \mathcal{GP}(\xi)$ and $\xi$-$\mathcal{G}{\rm pd} K\leq n-1$.

We define $\xi$-$\mathcal{G}$${\rm pd} M=n$ if $\xi$-$\mathcal{G}{\rm pd} M\leq n$ and $\xi$-$\mathcal{G}{\rm pd} M\nleq n-1$. If $\xi$-$\mathcal{G}$${\rm pd} M\neq n$ for all $n\geq 0$, we set $\xi$-$\mathcal{G}$${\rm pd} M=\infty$.

%Dually,  we can define the $\xi$-$\mathcal{G}$injective dimension  $\xi$-$\mathcal{G}{\rm id} A$ of an object $A\in\mathcal{C}$.

%We let $\widetilde{\mathcal{GP}}(\xi)$ denote the full subcategory of $\mathcal{C}$ whose objects are of finite $\xi$-$\mathcal{G}$projective dimension.

The following observation is useful in this section.

%
%\begin{definition}\label{df:3.2} Let $M\in\mathcal{C}$ be an object. A \emph{$\xi$-complete resolution} of $M$ is a diagram $\xymatrix@C=2em{\mathbf{T}\ar[r]^{\nu}&\mathbf{P}\ar[r]^{\pi}&M}$ of morphisms of complexes satisfying:
%  (1)  $\pi:\mathbf{P}\ra A$ is a $\xi$-projective resolution of $A$;
%  (2) $T$ is a complete $\mathcal{P}(\xi)$-exact complex;
%  (3) $\mu:\mathbf{T}\ra \mathbf{P}$ is a morphism such that $\mu_{i}$ $=$ {\rm id$_{T_{i}}$} for all $i\gg 0$.
%  A $\xi$-complete resolution is \emph{split} if $\mu_{i}$ has a section for all $i\in{\mathbb{Z}}$.
%\end{definition}
%
%Dually, for an object $B\in\mathcal{C}$ be an object with finite $\xi$-$\mathcal{G}$injective  dimension, we can construct a $\xi$-complete coresolution
%$\xymatrix@C=2em{B\ar[r]^{l}&\mathbf{I}\ar[r]^{\mu}&\mathbf{E}}$ of $B$, in which  $B\ra \mathbf{I}$ is a $\xi$-injective resolution of $B$
%and $\mathbf{E}$ is a complete $\xi$-injective coresolution.

%\begin{lem} \cite[Lemma 3.7(2)]{HZZ}\label{lem:3.11'} If $\xymatrix{A\ar[r]^{f}&B\ar[r]^{f}&C\ar@{-->}[r]^{\delta}&}$ is an $\mathbb{E}$-triangle in $\xi$, then $$\xymatrix{A\ar[r]^{\tiny\begin{bmatrix}0\\f\end{bmatrix}\qquad }&D\oplus B\ar[r]^{\tiny\begin{bmatrix}1&0\\0&g\end{bmatrix}}&
%D\oplus C\ar@{-->}[r]^{\ \ \ \ {\tiny\begin{bmatrix}0&1\end{bmatrix}}^*\delta}&}$$ is an $\mathbb{E}$-triangle.
%\end{lem}

\begin{lem}\label{lem:3.12'} Let $\xymatrix{A\ar[r]^{f}&B\ar[r]^{g}&C\ar@{-->}[r]^{\delta}&}$ be an $\mathbb{E}$-triangle in $\xi$. For any morphism $\alpha:B\ra D$, there exists an $\mathbb{E}$-triangle in $\xi$
$$\xymatrix{A\ar[r]^{\tiny\begin{bmatrix}-\alpha f\\f\end{bmatrix}\qquad }&D\oplus B\ar[r]^{\tiny\begin{bmatrix}1&\alpha\\0&g\end{bmatrix}}&
D\oplus C\ar@{-->}[r]^{\ \ \ \ {\tiny\begin{bmatrix}0&1\end{bmatrix}}^*\delta}&.}$$
\end{lem}
\begin{proof} It follows from \cite[Lemma 3.7(2)]{HZZ} that there exists an $\mathbb{E}$-triangle
$$\xymatrix{A\ar[r]^{\tiny\begin{bmatrix}0\\f\end{bmatrix}\qquad }&D\oplus B\ar[r]^{\tiny\begin{bmatrix}1&0\\0&g\end{bmatrix}}&
D\oplus C\ar@{-->}[r]^{\ \ \ \ {\tiny\begin{bmatrix}0&1\end{bmatrix}}^*\delta}&.}$$
Moreover, it is an $\mathbb{E}$-triangle in $\xi$ because $\xi$ is closed under base change. It is easy to check that the following is a commutative diagram.

$$\xymatrix{&D\oplus B\ar[dr]^{\tiny\begin{bmatrix}1&0\\0&g\end{bmatrix}}\ar[dd]_{\tiny\begin{bmatrix}1&-\alpha\\0&1\end{bmatrix}}^\cong&\\
A\ar[dr]_{\tiny\begin{bmatrix}-\alpha f\\f\end{bmatrix}}\ar[ur]^{\tiny\begin{bmatrix}0\\f\end{bmatrix}}&&D\oplus C\\
&D\oplus B\ar[ur]_{\tiny\begin{bmatrix}1&\alpha\\0&g\end{bmatrix}}&
}$$
Hence $\xymatrix{A\ar[r]^{\tiny\begin{bmatrix}-\alpha f\\f\end{bmatrix}\qquad }&D\oplus B\ar[r]^{\tiny\begin{bmatrix}1&\alpha\\0&g\end{bmatrix}}&
D\oplus C\ar@{-->}[r]^{\ \ \ \ {\tiny\begin{bmatrix}0&1\end{bmatrix}}^*\delta}&}$ is an $\mathbb{E}$-triangle in $\xi$.
\end{proof}

\begin{lem} \emph{(see \cite[Proposition 4.13(2)]{HZZ})}\label{lem:3.14'} Let $f:A\ra B$ and $g:B\ra C$ be morphisms in $\mathcal{C}$.  If the composition $gf$ is a $\xi$-deflation, then $g$ is a $\xi$-deflation.
\end{lem}

\begin{definition}\label{df:3.2} {\rm Let $M\in\mathcal{C}$ be an object. A \emph{$\xi$-complete resolution} of $M$ is a diagram $$\xymatrix@C=2em{\mathbf{T}\ar[r]^{\nu}&\mathbf{P}\ar[r]^{\pi}&M}$$ of morphisms of complexes satisfying:
  (1)  $\pi:\mathbf{P}\ra M$ is a $\xi$-projective resolution of $M$;
  (2) $\mathbf{T}$ is a complete $\xi$-projective resolution;
  (3) $\nu:\mathbf{T}\ra \mathbf{P}$ is a morphism such that $\nu_{i}$ $=$ {\rm id$_{T_{i}}$} for all $i\gg 0$.
  A $\xi$-complete resolution is \emph{split} if $\nu_{i}$ has a section {(}i.e., there exists a morphism $\eta_{i}:{P}_{i}\rightarrow {T}_{i}$ such that $\nu_{i}\eta_{i}={\rm id}_{{P}_{i}}${)} for all $i\in{\mathbb{Z}}$.}
\end{definition}

%\begin{cons}\label{Con:3.13} \end{cons}
%It follows from Constructions \ref{Con:3.10} and \ref{Con:3.13} that we have following result.
The following proposition characterizes when an object in $\mathcal{C}$ admits a $\xi$-complete resolution.

\begin{prop}\label{prop:3.15} Let $M$ be an object in $\mathcal{C}$. Then the following are equivalent for any non-negative integer $n$:
\begin{enumerate}
\item $\xi$-$\mathcal{G}{\rm pd} M\leq n$;
\item $M$ has a $\xi$-complete resolution $\xymatrix@C=2em{\mathbf{T}\ar[r]^{\nu}&\mathbf{P}\ar[r]^{\pi}&M}$ such that $\nu_{i}$ is an isomorphism for each $i\geq n$;
\item $M$ has a split $\xi$-complete resolution $\xymatrix@C=2em{\mathbf{S}\ar[r]^{\mu}&\mathbf{P}\ar[r]^{\pi}&M}$ such that $\mu_{i}$ is an isomorphism for each $i\geq n$.
\end{enumerate}
\end{prop}
\begin{proof} $(1)\Rightarrow(2)$. By hypothesis, we set $\xi$-$\mathcal{G}$${\rm pd} M\leq n$ and consider a $\xi$-projective resolution $\pi:\mathbf{P}\ra M$, where
$$\mathbf{P}:\cdots \ra P_n \ra \cdots\ra P_{1}\ra P_{0}\ra 0.$$
For any integer $i\in{\mathbb{Z}}$, there are $\mathbb{E}$-triangles $\xymatrix@C=2em{K_{i}\ar[r]&P_i\ar[r]&K_{i-1}\ar@{-->}[r]&}$ in $\xi$, where $K_i$ is called an $i$th $\xi$-syzygy of $M$. It follows from \cite[Proposition 5.2]{HZZ} that the $n$th $\xi$-syzygy $K_n$ is in ${\mathcal{GP}}(\xi)$. So there exists a  complete $\xi$-projective resolution of $K_n$
 $$\xymatrix@C=2em{\mathbf{Q}:\cdots\ar[r]&Q_1\ar[r]^{d_1}&Q_0\ar[r]^{d_0}&Q_{-1}\ar[r]&\cdots}$$
 with $\xymatrix@C=2em{K_{n}\ar[r]&Q_{n-1}\ar[r]&K'_{n-1}\ar@{-->}[r]&}$ in $\xi$. Then the $\xi$-projective resolution $\mathbf{P}$ of
$M$ and the  complete $\xi$-projective resolution $\mathbf{Q}$ of $K_n$ can be put together in a commutative
diagram
$$\xymatrix@C=2em{\mathbf{T}:=&\cdots\ar[r]&P_{n}\ar@{=}[d]\ar[r]&Q_{n-1}\ar[r]\ar[d]^{\nu_{n-1}}&\cdots\ar[r]&Q_{1}\ar[r]\ar[d]^{\nu_{1}}&
Q_{0}\ar[d]^{\nu_{0}}\ar[r]&Q_{-1}\ar[d]^{\nu_{-1}}\ar[r]&\cdots\\
\mathbf{P}:=&\cdots\ar[r]&P_{n}\ar[r]&P_{n-1}\ar[r]&\cdots\ar[r]&P_{1}\ar[r]&
P_{0}\ar[r]&0\ar[r]&\cdots.}$$
Since the upper row is $\mathcal{C}(-, \mathcal{P}(\xi))$-exact, the vertical maps can be constructed inductively,
starting from $K_n$. So $\xymatrix@C=2em{\mathbf{T}\ar[r]^{\nu}&\mathbf{P}\ar[r]^{\pi}&M}$ is a $\xi$-complete resolution of $M$ such that $\nu_{i}$ is an isomorphism for all $i\geq n$.

$(2)\Rightarrow(3)$. Note that $M$ has a $\xi$-complete resolution $\xymatrix@C=2em{\mathbf{T}\ar[r]^{\nu}&\mathbf{P}\ar[r]^{\pi}&M}$ such that $\nu_{i}$ is an isomorphism for each $i\geq n$ by (2). Then we have the following morphism  of $\mathbb{E}$-triangles in $\xi$
$$\xymatrix{K_n\ar@{=}[d]\ar[r]^{g_{n-1}'}&Q_{n-1}\ar[r]^{f_{n-1}'}\ar[d]^{\nu_{n-1}}&K_{n-1}'\ar[d]^{\omega_{n-1}}\ar@{-->}[r]^{\rho_{n-1}}&\\
K_n\ar[r]^{g_{n-1}}&P_{n-1}\ar[r]^{f_{n-1}}&K_{n-1}\ar@{-->}[r]^{\delta_{n-1}}&}$$
Moreover, for any integer $i<n$, we have the following morphism of $\mathbb{E}$-triangles in $\xi$
$$\xymatrix{K_{i}'\ar[d]^{\omega_{i}}\ar[r]^{g_{i-1}'}&Q_{i-1}\ar[r]^{f_{n-1}'}\ar[d]^{\nu_{i-1}}&K_{n-1}'\ar[d]^{\omega_{i-1}}\ar@{-->}[r]^{\rho_{i-1}}&\\
K_i\ar[r]^{g_{i-1}}&P_{i-1}\ar[r]^{f_{i-1}}&K_{i-1}\ar@{-->}[r]^{\delta_{i-1}}&}$$
By Lemma \ref{lem:3.12'}, there is an $\mathbb{E}$-triangle in $\xi$
$$\xymatrix{K_n\ar[r]^{\tiny\begin{bmatrix}-g_{n-1}\\g_{n-1}'\end{bmatrix}\ \ \ \ \ \ \ }&P_{n-1}\oplus Q_{n-1}\ar[r]^{\tiny \begin{bmatrix}1&\nu_{n-1}\\0&f_{n-1}'\end{bmatrix}\ \ \ }&P_{n-1}\oplus K_{n-1}'\ar@{-->}[r]^{\ \ \ \ \ \ \ \ \ \ \ \tiny\begin{bmatrix}0&1\end{bmatrix}^*\rho_{n-1}}&}$$
 Since the morphism $\tiny\begin{bmatrix}1&0\end{bmatrix}:P_{n-1}\oplus K_{n-1}'\rightarrow P_{n-1}$ is a split epimorphism, and  it is a $\xi$-deflation. Hence ${\tiny\begin{bmatrix}1&\nu_{n-1}\end{bmatrix}=\begin{bmatrix}1&0\end{bmatrix}\begin{bmatrix}1&\nu_{n-1}\\0&f_{n-1}'\end{bmatrix}}: P_{n-1}\oplus Q_{n-1}\rightarrow P_{n-1}$ is a $\xi$-deflation by \cite[Corollary 3.5]{HZZ}. It follows from \cite[Lemma 3.7(2)]{HZZ} that
 $\xymatrix{K_n\ar[r]^{\tiny\begin{bmatrix}0\\g_{n-1}'\end{bmatrix}\ \ \ \ \ \ \ }&P_{n-1}\oplus Q_{n-1}\ar[r]^{\tiny \begin{bmatrix}1&0\\0&f_{n-1}'\end{bmatrix}\ \ \ }&P_{n-1}\oplus K_{n-1}'\ar@{-->}[r]^{\ \ \ \ \ \ \ \ \ \ \ \tiny\begin{bmatrix}0&1\end{bmatrix}^*\rho_{n-1}}&}$ is an $\mathbb{E}$-triangle in $\xi$.
 Note that ${\tiny\begin{bmatrix}0\\1\end{bmatrix}^*\begin{bmatrix}f_{n-1}&\omega_{n-1}\end{bmatrix}}^*\delta_{n-1}=({\tiny\begin{bmatrix}f_{n-1}&\omega_{n-1}\end{bmatrix}
\begin{bmatrix}0\\1\end{bmatrix}})^*\delta_{n-1}=\omega_{n-1}^*\delta_{n-1}=\rho_{n-1}={\tiny\begin{bmatrix}0\\1\end{bmatrix}}^*
{\tiny\begin{bmatrix}0&1\end{bmatrix}}^*\rho_{n-1}$ and ${\tiny\begin{bmatrix}1\\0\end{bmatrix}^*\begin{bmatrix}f_{n-1}&\omega_{n-1}\end{bmatrix}}^*\delta_{n-1}=({\tiny\begin{bmatrix}f_{n-1}&\omega_{n-1}\end{bmatrix}
\begin{bmatrix}1\\0\end{bmatrix}})^*\delta_{n-1}=f_{n-1}^*\delta_{n-1}=0={\tiny\begin{bmatrix}1\\0\end{bmatrix}}^*
{\tiny\begin{bmatrix}0&1\end{bmatrix}}^*\rho_{n-1}$, which implies that ${\tiny\begin{bmatrix}f_{n-1}&\omega_{n-1}\end{bmatrix}^*\delta_{n-1}=\begin{bmatrix}0&1\end{bmatrix}^*\rho_{n-1}}$. Hence we have the following morphism of $\mathbb{E}$-triangles in $\xi$
$$\xymatrix{K_n\ar@{=}[d]\ar[r]^{\tiny\begin{bmatrix}0\\g_{n-1}'\end{bmatrix}\ \ \ \ \ \ \ }&P_{n-1}\oplus Q_{n-1}\ar[d]^{\tiny\begin{bmatrix}1&\nu_{n-1}\end{bmatrix}}\ar[r]^{\tiny \begin{bmatrix}1&0\\0&f_{n-1}'\end{bmatrix}\ \ \ }&P_{n-1}\oplus K_{n-1}'\ar[d]^{\tiny\begin{bmatrix}f_{n-1}&\omega_{n-1}\end{bmatrix}}\ar@{-->}[r]^{\ \ \ \ \ \ \ \ \ \ \ \tiny\begin{bmatrix}0&1\end{bmatrix}^*\rho_{n-1}}&\\
K_n\ar[r]^{g_{n-1}}&P_{n-1}\ar[r]^{f_{n-1}}&K_{n-1}\ar@{-->}[r]^{\delta_{n-1}}&
}$$
Since $K_{n-1}'$ is $\xi$-$\mathcal{G}$projective, the upper row is  $\mathcal{C}(-,\mathcal{P}(\xi))$-exact by \cite[Lemma 4.10(2)]{HZZ}. Since $${\tiny\begin{bmatrix}f_{n-1}&\omega_{n-1}\end{bmatrix}\begin{bmatrix}1&0\\0&f_{n-1}'\end{bmatrix}=f_{n-1}\begin{bmatrix}1&\nu_{n-1}\end{bmatrix}}$$ is a $\xi$-deflation, it follows from Lemma \ref{lem:3.14'} that the morphism ${\tiny\begin{bmatrix}f_{n-1}&\omega_{n-1}\end{bmatrix}}$ is also a $\xi$-deflation. Dual to \cite[Lemma 3.7(2)]{HZZ}, there exists an $\mathbb{E}$-triangle
$$\xymatrix{P_{n-1}\oplus K_{n-1}'\ar[r]^{\tiny\begin{bmatrix}1&0\\0&g_{n-2}'\end{bmatrix}\ \ \ }&P_{n-1}\oplus Q_{n-2}\ar[r]^{\tiny \ \ \ \ \begin{bmatrix}0&f_{n-2}'\end{bmatrix}}& K_{n-2}'\ar@{-->}[r]^{\ \  \tiny\begin{bmatrix}0\\1\end{bmatrix}_*\rho_{n-2}}&,}$$
which is also in $\xi$ because $\xi$ is closed under cobase change. Since ${\tiny\begin{bmatrix}f_{n-1}&\omega_{n-1}\end{bmatrix}_*\begin{bmatrix}0\\1\end{bmatrix}}_*\rho_{n-2}=({\tiny\begin{bmatrix}f_{n-1}&\omega_{n-1}\end{bmatrix}
\begin{bmatrix}0\\1\end{bmatrix}})_*\rho_{n-2}=\omega_{n-1*}\rho_{n-1}=\omega_{n-2}^*\delta_{n-2}$, we have the following morphism of $\mathbb{E}$-triangles in $\xi$
$$\xymatrix@C=4em{P_{n-1}\oplus K_{n-1}'\ar[d]_{\tiny\begin{bmatrix}f_{n-1}&\omega_{n-1}\end{bmatrix}\ \ \ }\ar[r]^{\tiny\begin{bmatrix}1&0\\0&g_{n-2}'\end{bmatrix}}
&P_{n-1}\oplus Q_{n-2}\ar[r]^{\ \ \tiny\begin{bmatrix}0&f_{n-2}'\end{bmatrix}}\ar[d]^{\tiny\begin{bmatrix}d_{n-1}^{\mathbf{P}}&\nu_{n-2}\end{bmatrix}}
&K_{n-2}'\ar[d]^{\omega_{n-2}}\ar@{-->}[r]^{\ \ \tiny\begin{bmatrix}0\\1\end{bmatrix}_*\rho_{n-2}}&\\
K_{n-1}\ar[r]^{g_{n-2}}&P_{n-2}\ar[r]^{f_{n-2}}&K_{n-2}\ar@{-->}[r]^{\delta_{n-2}}&}$$
Similarly, we have the following morphism of $\mathbb{E}$-triangles in $\xi$
$$\xymatrix@C=4em{P_{n-1}\oplus K_{n-1}'\ar[d]_{\tiny\begin{bmatrix}f_{n-1}&\omega_{n-1}\end{bmatrix}\ \ \ }\ar[r]^{\tiny\begin{bmatrix}0&0\\1&0\\0&g_{n-2}'\end{bmatrix}\ \ \ \ \ \ }
&P_{n-2}\oplus P_{n-1}\oplus Q_{n-2}\ar[r]^{\ \ \tiny\begin{bmatrix}1&0&0\\0&0&f_{n-2}'\end{bmatrix}}\ar[d]^{\tiny\begin{bmatrix}1&d_{n-1}^{\mathbf{P}}&\nu_{n-2}\end{bmatrix}}
&P_{n-2}\oplus K_{n-2}'\ar[d]^{\tiny\begin{bmatrix}f_{n-2}&\omega_{n-2}\end{bmatrix}}\ar@{-->}[r]^{\ \ \ \ \ \tiny\begin{bmatrix}0&1\end{bmatrix}^*\begin{bmatrix}0\\1\end{bmatrix}_*\rho_{n-2}}&\\
K_{n-1}\ar[r]^{g_{n-2}}&P_{n-2}\ar[r]^{f_{n-2}}&K_{n-2}\ar@{-->}[r]^{\delta_{n-2}}&}$$
where the morphisms ${\begin{bmatrix}1&d_{n-1}^{\mathbf{P}}&\nu_{n-2}\end{bmatrix}}$ and ${\begin{bmatrix}f_{n-2}&\omega_{n-2}\end{bmatrix}}$ are $\xi$-deflations and the upper row is a $\mathcal{C}(-,\mathcal{P}(\xi))$-exact $\mathbb{E}$-triangle in $\xi$. By proceeding in this manner, we set
$$S_i=\left\{
\begin{array}{cc}
P_i& i\geq n\\
P_{n-1}\oplus Q_{n-1}& i=n-1\\
P_i\oplus P_{i+1}\oplus Q_i& i<n-1
\end{array}\right.
$$
$$\mu_i=\left\{
\begin{array}{cc}
1& i\geq n\\
{\tiny\begin{bmatrix}1&\nu_{n-1}\end{bmatrix}}& i=n-1\\
{\tiny\begin{bmatrix}1&d_{i+1}^{\mathbf{P}}&\nu_{i}\end{bmatrix}}&0\leq i<n-1\\
0&i<0
\end{array}\right.
$$
Consequently, we get a commutative diagram
$$\xymatrix@C=2em{\mathbf{S}:=&\cdots\ar[r]&P_{n}\ar@{=}[d]\ar[r]&S_{n-1}\ar[r]\ar[d]^{\mu_{n-1}}&\cdots\ar[r]&S_{1}\ar[r]\ar[d]^{\mu_{1}}&
S_{0}\ar[d]^{\mu_{0}}\ar[r]&S_{-1}\ar[d]^{\mu_{-1}}\ar[r]&\cdots\\
\mathbf{P}:=&\cdots\ar[r]&P_{n}\ar[r]&P_{n-1}\ar[r]&\cdots\ar[r]&P_{1}\ar[r]&
P_{0}\ar[r]&0\ar[r]&\cdots.}$$
Note that every $S_i$ is $\xi$-projective, and $\mathbf{S}$ is obtained by pasting together those $\mathbb{E}$-triangles $$\xymatrix{K_n\ar[r]&P_{n-1}\oplus Q_{n-1}\ar[r]&P_{n-1}\oplus K_{n-1}'\ar@{-->}[r]&}$$ and $$\xymatrix{P_i\oplus K_i'\ar[r]&P_{i-1}\oplus P_i\oplus Q_{i-1}\ar[r]&P_{i-1}\oplus K_{i-1}'\ar@{-->}[r]&}$$ for all $i<n$, then the complex $\mathbf{S}$ is $\xi$-exact and $\mathcal{C}(-,\mathcal{P}(\xi))$-exact. Hence we get the desired $\xi$-complete resolution $\xymatrix{\mathbf{S}\ar[r]^{\mu}&\mathbf{P}\ar[r]^\pi&M}$ such that $\mu_i$ is a $\xi$-deflation for every $i\in \mathbb{Z}$. So  each $\mu_i$ has a section and $\xymatrix@C=2em{\mathbf{S}\ar[r]^{\mu}&\mathbf{P}\ar[r]^{\pi}&M}$ is a split $\xi$-complete resolution such that $\mu_{i}$ is an isomorphism for each $i\geq n$, as desired.

$(3)\Rightarrow(1)$ is trivial.
\end{proof}

As a similar argument to that of \cite[Lemma 3.2]{AS2}, we have the following result.

\begin{lem}\label{lem:4.4}
Let $\xymatrix@C=2em{\mathbf{T}\ar[r]^{\nu}&\mathbf{P}\ar[r]^{\pi}&M}$ and $\xymatrix@C=2em{\mathbf{T}'\ar[r]^{\nu'}&\mathbf{P}'\ar[r]^{\pi'}&M'}$ be  $\xi$-complete resolutions of $M$ and $M'$, respectively. For each morphism $\mu: M\to M'$, there exists a morphism $\overline{\mu}$, unique up to homotopy, making the right-hand square of the following diagram
$$
\xymatrix{\mathbf{T}\ar[r]^{\nu}\ar[d]^{\widetilde{\mu}}&\mathbf{P}\ar[r]^{\pi}\ar[d]^{\overline{\mu}}&M\ar[d]^{\mu}\\
\mathbf{T}'\ar[r]^{\nu'}&\mathbf{P}'\ar[r]^{\pi'}&M'}
$$
commutative, and for each choice of $\overline{\mu}$, there exists a morphism $\widetilde{\mu}$, unique up to homotopy, making the left-hand square commute up to homotopy. In particular, if \emph{$\mu=\textrm{id}_{M}$}, then $\widetilde {\mu}$ and $\overline{\mu}$ are homotopy equivalences.
\end{lem}

%By the proof of $(2)\Rightarrow(3)$ in Proposition \ref{prop:3.15}, we have the following result.
%
%\begin{cor}\label{lem:costruct-of-homotopy} Let $\xymatrix@C=2em{\mathbf{T}\ar[r]^{\nu}&\mathbf{P}\ar[r]^{\pi}&M}$ be a $\xi$-complete resolution of $M$ with $\nu_{i}$ being an isomorphism for each $i\geq n$. Then there exist a split $\xi$-complete resolution $\xymatrix@C=2em{\mathbf{S}\ar[r]^{\mu}&\mathbf{P}\ar[r]^{\pi}&M}$ of $M$ and a homotopy equivalence $\alpha:\mathbf{T}\ra \mathbf{S}$ such that $\nu=\mu \alpha$ and {\rm$\alpha_{i}=\textrm{id}_{\mathbf{T}_{i}}$} for all $i\geq n$.
%\end{cor}

We are now in a position to prove the following result, which gives a general technique for computing $\xi$-complete cohomology of objects with finite $\xi$-$\mathcal{G}$projective dimension.

\begin{thm}\label{thm2} Let $M$ and $N$ be objects in $\mathcal{C}$. If $\xymatrix@C=2em{\mathbf{T}\ar[r]^{\nu}&\mathbf{P}\ar[r]^{\pi}&M}$ is a $\xi$-complete resolution of $M$, then for any integer $n$, there exists  an isomorphism
$$\widetilde{\xi{\rm xt}}_{\mathcal{P}}^n(M,N)\cong H^n(\mathcal{C}(\mathbf{T},N)).$$
\end{thm}
\begin{proof} By Proposition \ref{prop:3.15}, there exists a split $\xi$-complete resolution $\xymatrix@C=2em{\mathbf{S}\ar[r]^{\mu}&\mathbf{P}\ar[r]^{\pi}&M}$ of $M$ such that $\mu_{i}$ is an isomorphism for each $i\geq m$. Thus $\mathbf{T}$ and $\mathbf{S}$ are homotopy equivalences by Lemma \ref{lem:4.4}, and hence  $H^n(\mathcal{C}(\mathbf{S},N)\cong H^n(\mathcal{C}(\mathbf{T},N)$ for any integer $n$. Next we show that $\widetilde{\xi{\rm xt}}_{\mathcal{P}}^n(M,N)\cong H^n(\mathcal{C}(\mathbf{S},N))$ for any integer $n$.

Note that $\xymatrix@C=2em{\mathbf{S}\ar[r]^{\mu}&\mathbf{P}\ar[r]^{\pi}&M}$ is a split $\xi$-complete resolution. Then we have a sequence $\xymatrix@C=2em{\mathbf{X}\ar[r]^{\lambda}&\mathbf{S}\ar[r]^{\mu}&\mathbf{P}}$ of complexes such that $\xymatrix@C=2em{X_{i}\ar[r]^{\lambda_i}&S_{i}\ar[r]^{\mu_i}&P_i\ar@{-->}[r]^{\delta_i}&}$ is a split $\mathbb{E}$-triangle for each $i\in\mathbb{Z}$.
  Let $\xymatrix@C=2em{\rho:\mathbf{Q}\ar[r]& N}$ be a $\xi$-projective resolution of $N$. Applying the functor $\mathcal{C}(-,\mathbf{Q})$ to the sequence above, we have the following commutative diagram with exact rows and columns:
$$\xymatrix{
   & 0\ar[d]& 0\ar[d] & 0 \ar[d] & \\
  0 \ar[r] & \overline{\mathcal{C}}(\mathbf{P},\mathbf{Q}) \ar[d] \ar[r] & {\mathcal{C}}(\mathbf{P},\mathbf{Q}) \ar[d] \ar[r] & \widetilde{\mathcal{C}}(\mathbf{P},\mathbf{Q}) \ar[d] \ar[r] & 0 \\
  0 \ar[r] & \overline{\mathcal{C}}(\mathbf{S},\mathbf{Q}) \ar[d] \ar[r] &{\mathcal{C}}(\mathbf{S},\mathbf{Q}) \ar[d] \ar[r] &\widetilde{\mathcal{C}}(\mathbf{S},\mathbf{Q}) \ar[d] \ar[r] &0 \\
 0\ar[r]& \overline{\mathcal{C}}(\mathbf{X},\mathbf{Q}) \ar[d] \ar[r] & {\mathcal{C}}(\mathbf{X},\mathbf{Q}) \ar[d] \ar[r] & \widetilde{\mathcal{C}}(\mathbf{X},\mathbf{Q}) \ar[d] \ar[r] &0 \\
   & 0 &0 &0. &    }
  $$
Since $\mathbf{X}$ is bounded above,  $\overline{\mathcal{C}}(\mathbf{X},\mathbf{Q})={\mathcal{C}}(\mathbf{X},\mathbf{Q})$. For each integer $n\in\mathbb{Z}$, we have
$$H^n(\widetilde{\mathcal{C}}(\mathbf{P},\mathbf{Q}))\cong H^n(\widetilde{\mathcal{C}}(\mathbf{S},\mathbf{Q})).$$
Note that $\mathbf{S}$ is a $\xi$-exact complex such that $\mathcal{C}(\mathbf{S},Q)$ is exact for all $Q\in\mathcal{P}(\xi)$. Then $\overline{\mathcal{C}}(\mathbf{S},\mathbf{Q})$ is exact by \cite[Proposition A.2(b)]{CCLP}. Thus
$$H^n({\mathcal{C}}(\mathbf{S},\mathbf{Q}))\cong H^n(\widetilde{\mathcal{C}}(\mathbf{S},\mathbf{Q})), \forall n\in\mathbb{Z}.$$
Note that $\xymatrix@C=2em{\rho:\mathbf{Q}\ar[r]&N}$ is a $\xi$-projective resolution of $N$. Then for any $\xi$-projective object $P$,
 \noindent$\xymatrix@C=2em{\mathcal{C}(P,\rho):\mathcal{C}(P,\mathbf{Q})\ar[r]&\mathcal{C}(P,N)}$ is a quasi-isomorphism. It is not hard to check that the morphism $\xymatrix@C=2em{{\mathcal{C}}(\mathbf{S},\rho):{\mathcal{C}}(\mathbf{S},\mathbf{Q})\ar[r]&{\mathcal{C}}(\mathbf{S},N)}$ is a quasi-isomorphism and the proof is similar to that of Lemma \ref{lem:3.8}(1). So $$H^n({\mathcal{C}}(\mathbf{S},\mathbf{Q}))\cong H^n({\mathcal{C}}(\mathbf{S},N)), \forall n\in\mathbb{Z}.$$
Hence for each integer $n$, we have
$$H^n(\widetilde{\mathcal{C}}(\mathbf{P},\mathbf{Q}))\cong H^n(\widetilde{\mathcal{C}}(\mathbf{S},\mathbf{Q}))\cong H^n({\mathcal{C}}(\mathbf{S},\mathbf{Q}))\cong H^n({\mathcal{C}}(\mathbf{S},N)).$$
So  $\widetilde{{\rm \xi}xt}_{\mathcal{P}}^n(M,N)\cong H^n({\mathcal{C}}(\mathbf{S},N)) \cong H^n(\mathcal{C}(\mathbf{T},N) \ for \ all \  n\in\mathbb{Z}.$ This completes the proof.
\end{proof}
\begin{rem}\label{rem:4.4} Assume that $\mathcal{C}$ is a triangulated category and $\xi$ is a proper class  of triangles which is closed under suspension (see \cite[Section 2.2]{Bel1}). Let $M$ be an object in $\mathcal{C}$ with finite $\xi$-$\mathcal{G}$projective dimension and $N$  an object in $\mathcal{C}$. By Theorem \ref{thm2}, one can check that for any integer $n$, the complete cohomology group $\widetilde{{\xi}{\rm xt}}_{\mathcal{P}}^n(M,N)$ defined here is exactly the Tate cohomology group $\widehat{\xi{\rm xt}}_{\mathcal{P}}^n(M,N)$ defined by Asadollahi and Salarian in \cite{AS2}.
\end{rem}

Let $M$ be an object in $\mathcal{C}$. It is shown that $\xi$-$\mathcal{G}{\rm pd} M\leq \xi$-${\rm pd} M$, and the equality holds if $\xi$-${\rm pd} M<\infty$ (see \cite[Proposition 5.4]{HZZ}).  It seems natural to ask when $\xi$-$\mathcal{G}{\rm pd} M= \xi$-${\rm pd} M$ provided that $\xi$-$\mathcal{G}{\rm pd} M<\infty$. Motivated by this, we have the following result which is a consequence of Theorems  \ref{thm1} and \ref{thm2}.

%The first application of Theorem \ref{thm:1.1} is the next result characterizing when $\textrm{pd}_R(M)=\mathcal{GC}$-dim$M$ provided that
%$\mathcal{GC}$-dim$M$ $<\infty$. See ?

\begin{cor}\label{cor:3.3} If $\mathcal{P}(\xi)$ is a generating subcategory of $\mathcal{C}$ and $M$ is an object in $\mathcal{C}$ with finite $\xi$-$\mathcal{G}$projective dimension, then the following are equivalent:
\begin{enumerate}
\item $\xi$-$\mathcal{G}{\rm pd} M= \xi$-${\rm pd} M$;

\item $\widetilde{{\rm \xi xt}}_{\mathcal{P}}^i(M,N)=0$ for any integer $i\in\mathbb{Z}$ and any object $N$ in $\mathcal{C}$;

\item $\widetilde{{\rm \xi xt}}_{\mathcal{P}}^i(M,N)=0$ for some integer $i\in\mathbb{Z}$ and any object $N$ in $\mathcal{C}$.

\end{enumerate}
\end{cor}
\begin{proof}  $(1)\Rightarrow(2)$ holds by Theorem \ref{thm1}.

$(2)\Rightarrow(3)$ is trivial.

$(3)\Rightarrow(1)$.  Note that $M$ is an object with finite $\xi$-$\mathcal{G}$projective dimension by hypothesis. Then there exists an integer $m$ and a split $\xi$-complete resolution $\xymatrix@C=2em{\mathbf{S}\ar[r]^{\mu}&\mathbf{P}\ar[r]^{\pi}&M}$ of $M$ such that $\mu_{i}$ is an isomorphism for each $i \geq m$. In fact $\mathbf{S}$ is a $\xi$-exact complex such that there is a $\mathcal{C}(-,\mathcal{P}(\xi))$-exact $\mathbb{E}$-triangle in $\xi$ $$\xymatrix@C=2em{K_{i+1}\ar[r]^{g_i}&S_i\ar[r]^{f_i}&K_i\ar@{-->}[r]^{\delta_i}&}$$  for each integer $i$ and $d^{\mathbf{S}}_i=g_{i-1}f_i$.
By (3), there exists an integer $n$ such that $\widetilde{{\rm \xi xt}}_{\mathcal{P}}^{n+1}(M,N)=0$ for any object $N$ in $\mathcal{C}$. It follows from Theorem \ref{thm2} that the following sequence
$$\xymatrix@C=35pt@R=35pt{{\mathcal{C}}(S_{n-1},K_n)\ar[rr]^{\mathcal{C}(d^{\mathbf{S}}_{n},K_n)}
&&{\mathcal{C}}(S_{n},K_n)\ar[rr]^{\mathcal{C}(d^{\mathbf{S}}_{n+1},K_n)}&& {\mathcal{C}}(S_{n+1},K_n)
    }$$
is exact. Note that $\mathcal{C}(d^{\mathbf{S}}_{n+1},K_n)(f_n)=f_nd^{\mathbf{S}}_{n+1}=0$. Then there exists a morphism $h:S_{n-1}\ra K_{n}$ such that $f_{n}=h d^{\mathbf{S}}_{n}= hg_{n-1}f_{n}$. Let $P$ be any object in $\mathcal{P}(\xi)$. Then ${\mathcal{C}}(P,f_n)$ is epic. It follows that ${\mathcal{C}}(P,h){\mathcal{C}}(P,g_{n-1})=\textrm{id}_{{\mathcal{C}}(P,S_n)}$. Hence ${\mathcal{C}}(P,S_{n-1})\cong {\mathcal{C}}(P, K_{n}\oplus K_{n-1})$. Since $\mathcal{P}(\xi)$ is a generating subcategory of $\mathcal{C}$ by hypothesis, we get that $S_{n-1}\cong K_{n}\oplus K_{n-1}$. It follows that $K_{n}\in{\mathcal{P}(\xi)}$, so $K_{i}\in{\mathcal{P}(\xi)}$ for all $i\geq n$. Let $s$ be an integer such that $s\geq$ max\{$n$, $m$\}. Then $K_{s}\in{\mathcal{P}(\xi)}$ and $\xi$-${\rm pd} M \leq s$. So (1) holds by \cite[Proposition 5.4]{HZZ}. This completes the proof.
\end{proof}

We end this paper with the following remark.
\begin{rem} Using arguments analogous to the ones employed in this section, one may show that  an object $M$ in  an extriangulated category $(\mathcal{C}, \mathbb{E}, \mathfrak{s})$  has finite $\xi$-$\mathcal{G}$injective dimension if and only if $M$ admits a (split) $\xi$-complete coresolution, which can help us to give a general technique for computing complete cohomology of objects with finite $\xi$-$\mathcal{G}$injective dimension. The details are left to the reader.
\end{rem}

\bigskip

\renewcommand\refname
\vspace{4mm}
\small

\hspace{-1.2em}\textbf{Jiangsheng Hu}\\
School of Mathematics and Physics, Jiangsu University of Technology,
 Changzhou 213001, China\\
E-mail: jiangshenghu@jsut.edu.cn\\[1mm]
\textbf{Dongdong Zhang}\\
Department of Mathematics, Zhejiang Normal University,
 Jinhua 321004, China\\
E-mail: zdd@zjnu.cn\\[1mm]
\textbf{Tiwei Zhao}\\
School of Mathematical Sciences, Qufu Normal University, Qufu 273165, China\\
E-mail: tiweizhao@qfnu.edu.cn\\[1mm]
\textbf{Panyue Zhou}\\
College of Mathematics, Hunan Institute of Science and Technology, Yueyang 414006, China\\
E-mail: panyuezhou@163.com

\end{document}